\documentclass[11pt]{article}
\usepackage{amsfonts,amssymb,amscd,amsmath,enumerate,verbatim,calc, times,url}
\usepackage{amsthm}
\usepackage{psfrag}
\usepackage{latexsym}
\usepackage{mdwlist}

\usepackage[dvips]{graphicx}
\usepackage{amsfonts}
\usepackage{amsthm}
\usepackage{epsfig}

\theoremstyle{plain}
\newtheorem{theorem}{Theorem}[section]
\newtheorem{lemma}[theorem]{Lemma}
\newtheorem{proposition}[theorem]{Proposition}
\newtheorem{corollary}[theorem]{Corollary}
\theoremstyle{definition}
\newtheorem{definition}[theorem]{Definition}

\newtheorem{example}[theorem]{Example}

\newcommand{\xs}{x_1,\ldots,x_n}                %x_1,...,x_n
\newcommand{\vs}{v_1,\ldots,v_n}                %v_1,...,v_n
\newcommand{\Fs}{F_1,\ldots,F_q}                %F_1,...,F_q

\newcommand{\dimn}{ {\rm{dim}} \ }
\newcommand{\F}{{\mathcal{F}}}                  %Facet ideal of a complex
\newcommand{\C}{{\mathcal{C}}}       
\newcommand{\GG}{{\mathcal{G}}}                  %generators of ideal
\newcommand{\D}{\Delta}                         %Delta
\newcommand{\lcm}{{\mathop{\rm{lcm}}}}          %least common multiple
\newcommand{\ndiv}{\not |}              
\newcommand{\st}{\ | \ }                        % such that: | 
\newcommand{\tuple}[1]{\langle #1 \rangle}      % < ... >
\newcommand{\rmv}[1]{\setminus \langle #1\rangle}% remove a facet 
\newcommand{\void}[1]{}
\newcommand{\conn}{\mbox{conn}}
\newcommand{\oconn}{\overline{\mbox{conn}}}
\newcommand{\pd}{\mbox{projdim}}
\newcommand{\reg}{\mbox{reg}}
\newcommand{\cocoa}{\mbox{\rm C\kern-.13em o\kern-.07 em C\kern-.13em o\kern-.15em A}} 
\newcommand{\cocoax}{\mbox{C\kern-.13em o\kern-.07 em C\kern-.13em o\kern-.15em A}} 
\newcommand{\cocoal}{\mbox{\rm C\kern-.13em o\kern-.07 em C\kern-.13em o\kern-.15emA\kern-.1em L}}

\newcommand{\todo}[1]{\vspace{5 mm}\par \noindent
\marginpar{\textsc{ToDo}}
\framebox{\begin{minipage}[c]{0.95 \textwidth}
\tt #1 \end{minipage}}\vspace{5 mm}\par}

\renewcommand{\todo}[1]{}

\newcommand{\idiot}[1]{\vspace{5 mm}\par \noindent
\framebox{\begin{minipage}[c]{0.95 \textwidth}
\tt #1 \end{minipage}}\vspace{5 mm}\par}

\renewcommand{\idiot}[1]{}

\newcommand{\sm}{\setminus}

\renewcommand{\leq}{\leqslant}          
\renewcommand{\geq}{\geqslant}

\date{\today}

\author{Sara Faridi\thanks{Department of Mathematics and Statistics,
    Dalhousie University, Halifax, Canada, faridi@mathstat.dal.ca.
    Research supported by NSERC. The author acknowledges the
    hospitality of MSRI in Berkeley, CA where part of this work was
    completed.}}

\title{\Large A Good Leaf Order on Simplicial Trees}
\begin{document}

\maketitle
\begin{abstract} Using the existence of a good leaf in every simplicial
 tree, we order the facets of a simplicial tree in order to find
 combinatorial information about the Betti numbers of its facet ideal.
 Applications include an Eliahou-Kervaire splitting of the ideal, as
 well as a refinement of a recursive formula of H\`a and Van Tuyl for
 computing the graded Betti numbers of simplicial trees.
 \end{abstract}

\begin{quotation}
 \begin{center} 
  \emph{Dedicated to Tony Geramita for his many contributions to Mathematics}
 \end{center}
\end{quotation}

%%%%%%%%%%%%%%%%%%%%%%%%%%%%%%%%%%%%%%%%%%%%%%%%%%%%%%%%%%%%
%%%%%%%%%%%%%%%%%%%%%%%%%%%%%%%%%%%%%%%%%%%%%%%%%%%%%%%%%%%%
\section{Introduction} 
 
Given a monomial ideal $I$ in a polynomial ring $R=k[\xs]$ over a
field $k$, a {\bf minimal free resolution} of $I$ is an exact sequence
of free $R$-modules
$$0\rightarrow {\displaystyle
  \bigoplus_{d}}R(-d)^{\beta_{p,d}}\rightarrow\cdots{\displaystyle
  \rightarrow\bigoplus_{d}}R(-d)^{\beta_{0,d}}\rightarrow I
\rightarrow 0$$ of $R /I $ in which $R(-d)$ denotes the graded free
module obtained by shifting the degrees of elements in $R$ by $d$. The
numbers $\beta_{i,d}$, which we shall refer to as the $i$-th
$\mathbb{N}$-\textbf{graded Betti numbers} of degree $d$ of $R /I$,
are independent of the choice of graded minimal finite free
resolution.

Questions about Betti numbers - including when they vanish and when
they do not, what bounds they have, how they relate to the base field
$k$ and what are the most effective ways to compute them - are of
particular interest in combinatorial commutative algebra. Via a method
called polarization~\cite{Fr}, it turns out that it is enough to
consider such questions for square-free monomial ideals~\cite{GPW};
i.e. a monomial ideal in which the generators are square-free monomials.

To a square-free monomial ideal $I$ one can associate a unique
simplicial complex called its facet complex. Conversely, every
simplicial complex has a unique monomial ideal assigned to it called
its facet ideal~\cite{F1}. Simplicial trees~\cite{F1} and related
structures were developed as a class of simplicial complexes that
generalize graph-trees, so that their facet ideals have similar
properties to those of edge ideals of graphs discovered in a series of
works by Villarreal and his coauthors~\cite{V}.

This paper offers an order on the monomials generating the facet
ideal of a simplicial tree which uses the existence of a ``good
  leaf'' in every simplicial tree~\cite{HHTZ}. This order in itself
is combinatorially interesting and useful, but it turns out that it
also produces a ``splitting''~\cite{EK} of the facet ideal of a tree which
gives bounds on the Betti numbers of the ideal. 

Our good leaf order also makes it possible to refine a recursive
formula of H\`a and Van Tuyl~\cite{HV} for computing Betti numbers of
facet ideals of simplicial trees, and to apply it to classes of trees
with strict good leaf orders. The idea here is that a good leaf order
will split an ideal to some extent, and within each one of these split
pieces, one can apply H\`a and Van Tuyl's formula quite efficiently if
the order is strict.

%%%%%%%%%%%%%%%%%%%%%%%%%%%%%%%%%%%%%%%%%%%%%%%%%%%%%%%%%%%%
%%%%%%%%%%%%%%%%%%%%%%%%%%%%%%%%%%%%%%%%%%%%%%%%%%%%%%%%%%%%
\section{Simplicial complexes, trees and forests}

 \begin{definition}[simplicial complexes] 
   A {\bf simplicial complex} $\D$ over a set of vertices $V(\D)=\{
   \vs \}$ is a collection of subsets of $V(\D)$, with the property
   that $\{ v_i \} \in \D$ for all $i$, and if $F \in \D$ then all
   subsets of $F$ are also in $\D$. An element of $\D$ is called a
   {\bf face} of $\D$, and the {\bf dimension} of a face $F$ of $\D$
   is defined as $|F| -1$, where $|F|$ is the number of vertices of
   $F$.  The faces of dimensions 0 and 1 are called {\bf vertices} and
   {\bf edges}, respectively, and $\dimn \emptyset =-1$.  The maximal
   faces of $\D$ under inclusion are called {\bf facets} of $\D$. The
   dimension of the simplicial complex $\D$ is the maximal dimension
   of its facets.  A {\bf subcollection} of $\D$ is a simplicial
   complex whose facets are also facets of $\D$.

   A simplicial complex $\D$ is {\bf connected} if for every pair of
   facets $F$, $G$ of $\D$, there exists a sequence of facets
   $F_1,\ldots,F_r$ of $\D$ such that $F_1=F$, $F_r=G$ and $F_s \cap
   F_{s+1} \neq \emptyset$ for $1 \leq s <r$.
   
   We use the notation $\tuple{\Fs}$ to denote the simplicial complex
   with facets $\Fs$, and we call it the simplicial complex {\bf
     generated by} $\Fs$.  By {\bf removing the facet} $F_i$ from $\D$
   we mean the simplicial complex $\D \rmv{F_i}$ which is generated by 
   $\{\Fs\} \sm \{F_i\}$.

 \end{definition}

\begin{definition}[Leaf, joint, simplicial trees and forests~\cite{F1}] 
A facet $F$ of a simplicial complex $\D$ is called a {\bf leaf} if
either $F$ is the only facet of $\D$, or $F$ intersects $\D \rmv{F}$ in
a face of $ \D \rmv{F}$. If $F$ is a leaf and $\D$ has more than one
facet, then for some facet $G \in \D \rmv{F}$ we have $F \cap H
\subseteq G$ for all $H\in \D \rmv{F}$.  Such a facet $G$ is
called a {\bf joint} of $F$.

A simplicial complex $\D$ is a {\bf simplicial forest} if every
nonempty subcollection of $\D$ has a leaf.  
 A connected simplicial forest is called a {\bf simplicial tree}.
\end{definition}

It follows easily from the definition that a leaf must always contain
at least one {\bf free vertex}, that is a vertex that belongs to no
other facet of $\D$.

\begin{example} The facets $F_0, F_2$ and $F_4$ are all leaves 
of the simplicial tree in Figure~\ref{f:glo}. The first two
have $F_1$ as a joint and $F_4$ has $F_3$ as a joint.
\end{example}

\begin{figure}
 \centering \includegraphics{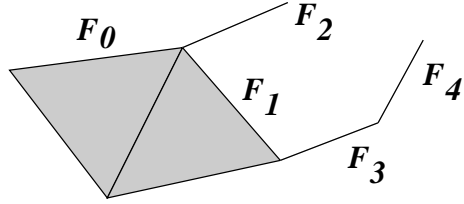}
\caption{Good leaves}\label{f:glo}
\end{figure}

\begin{definition}[Good leaf~\cite{Z,CFS}]\label{d:good-leaf} A facet $F$ of
a simplicial complex $\D$ is called a {\bf good leaf} of $\D$ if $F$
is a leaf of every subcollection of $\D$ which contains $F$.
\end{definition}

All leaves of the simplicial tree in Figure~\ref{f:glo} are good
leaves. Figure~\ref{f:nglo} contains an example of a leaf $F$ in a
simplicial tree which is not a good leaf: if we remove the facet $G$
then $F$ is no longer a leaf.

\begin{figure}
 \centering \includegraphics{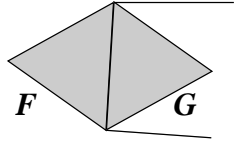}
\caption{A leaf that is not a good leaf}\label{f:nglo}
\end{figure}

Good leaves were studied in~\cite{Z} and then independently
in~\cite{CFS} (where they were called ``reducible leaves''). In both
sources the existence of such a leaf in every tree was conjectured but
not proved; the proof came later, using incidence matrices.

\begin{theorem}[\cite{HHTZ}] Every simplicial tree contains a good leaf.
\end{theorem}

\begin{definition}[Facet ideal, facet complex~\cite{F1}] Let $\D$ be a
  simplicial complex with vertex set $\{\xs\}$, and let $R=k[\xs]$ be
  a polynomial ring over a field $k$ with variables corresponding to
  the vertices of $\D$.  The {\bf facet ideal} of $\D$, denoted by
  $\F(\D)$, is an ideal of $R$ whose generators are monomials, each of
  which is the products of the variables labeling the vertices of a
  facet of $\D$. Given a square-free monomial ideal $I$ in $R$, the
  {\bf facet complex of $I$} is the simplicial complex whose facets
  are the set of variables appearing in each monomial generator of
  $I$.
\end{definition}

\begin{example} If $I=(xy,yzu,xz)$ is a monomial ideal in
  $R=k[x,y,z,u]$, its facet complex is the simplicial complex
  $\D=\tuple{\{x,y\},\{y,z,u\},\{x,z\}}$. Similarly $I$ is the facet
  ideal of $\D$.
\end{example}

It is clear from the definition and example that every square-free
monomial ideal has a unique facet complex, and every simplicial
complex has a unique facet ideal.  Because of this one-to-one
correspondence we often abuse notation and use facets and monomials
interchangeably. For example we say $F\cup G = \lcm(F,G)$ to imply the
union of two facets $F$ and $G$ or the least common multiple of two
monomials [corresponding to the facets] $F$ and $G$.

Trees behave well under localization:

\begin{lemma}[Localization of a tree is a
  forest~\cite{F1}]\label{l:localization}
  Let $\Delta$ be a simplicial tree with vertices $\xs$, and let $I$
  be the facet ideal of $\D$ in the polynomial ring $R=k[\xs]$ where
  $k$ is a field. Then for any prime ideal $p$ of $R$, $I_p$ is the
  facet ideal of a simplicial forest. \end{lemma}

For a simplicial complex $\D$ with a facet $F$, we use the notation
$\D_{\overline{F}}$ for facet complex of the localization $\F(\D)$ at the ideal
generated by the the complement of the facet $F$.
%%%%%%%%%%%%%%%%%%%%%%%%%%%%%%%%%%%%%%%%%%%%%%%%%%%%%%%%%%%%
%%%%%%%%%%%%%%%%%%%%%%%%%%%%%%%%%%%%%%%%%%%%%%%%%%%%%%%%%%%%

\section{Good leaf orders} 

From its definition it is immediate that a good leaf $F_0$ of a tree
$\D$ induces an order $F_0, F_1,\ldots,F_q$ on the facets of $\D$ so
that $$F_0\cap F_1 \supseteq F_0 \cap F_2 \supseteq \cdots \supseteq
F_0\cap F_q.$$ Our goal in this section is to demonstrate that this
order can be refined so that $\D$ is built leaf by leaf starting from
the good leaf $F_0$. In other words, the order can be written so that
for $i\leq q$, $F_i$ is a leaf of $\D_i=\tuple{F_0,\ldots,F_i}$. Such
an order on the facets of $\D$ will be called a {\bf good leaf order}
on $\D$.

\begin{example} Let $\D$ be the simplicial tree in Figure~\ref{f:glo}. 
Then $F_0$ is a good leaf and the labeling of facets $F_0,\ldots,F_4$
is a good leaf order on $\D$, since $F_0 \cap F_1 \supseteq \cdots
\supseteq F_0 \cap F_4$. Note that even though $F_0 \cap F_1 \supseteq
F_0 \cap F_2 \supseteq F_0 \cap F_4 \supseteq F_0 \cap F_3$, this
latter order $F_0,F_1,F_2,F_4.F_3$ is not a good leaf order since
$F_3$ is not a leaf of $\D$.
\end{example}

We show that every simplicial tree (forest) has a good leaf order.

\begin{lemma}\label{l:three-indices} Suppose $\D= \langle F,
  G, H \rangle$ is a simplicial tree with $F \cap G \not \subseteq H$
  and $F \cap H \not \subseteq G$. Then $G$ and $H$ are the leaves of
  the tree $\D$ and $F$ is the common joint so that $G \cap H
  \subseteq F$.
\end{lemma}

      \begin{proof}  
        If $F$ is a leaf, then either $F \cap G \subseteq H$ or $F \cap
        H \subseteq G$.  Either case is a contradiction, so the two
        leaves of the tree have to be $G$ and $H$.  If $H$ is a joint
        of the leaf $G$ then $F \cap G \subseteq H$ which is again a
        contradiction, so $F$ is the joint of $G$. Similarly, $F$ is
        the joint of $H$, and we have $G \cap H \subseteq F$.
      \end{proof}

\begin{proposition}[First step to build good leaf order]\label{p:glo} 
Let $\D$ be a simplicial tree with a good leaf $F_0$ and good leaf
order $$F_0\cap F_1 \supseteq F_0 \cap F_2 \supseteq \cdots \supseteq
F_0\cap F_q.$$ Let $1 \leq a \leq q$ and $0\leq b < a$ and
$$F_0 \cap F_{a-b-1} \supsetneq F_0\cap F_{a-b}=\cdots= F_0 \cap F_a.$$
Then one of $F_{a-b},\ldots,F_{a}$ is
a leaf of $\langle F_0,\ldots, F_{a}\rangle$.
\end{proposition}

\begin{proof} Let $\Gamma=\langle F_0, \ldots, F_a\rangle$.  The
   subcollection $\Omega=\tuple{F_0,\ldots,F_{a-b-1}}$ of $\Gamma$ is
   connected as all facets have nonempty intersection with
   $F_0$. If $\Gamma$ is disconnected, $\Omega$ will be
   contained in one of the connected components of $\Gamma$, and there
   will be another connected component $\Sigma$ whose facets are from
   $F_{a-b},\ldots, F_a$. Since $\Sigma$ is a subcollection of a tree,
   it must have a leaf, and that leaf will be a leaf of $\Gamma$ as
   well. So one of $F_{a-b},\ldots,F_{a}$ will be a leaf of $\Gamma$.

  We now assume that $\Gamma$ is connected and proceed by induction on
  $a$ to prove our claim.  If $a=1$ then clearly $F_1$ is a leaf of the
  tree $\Gamma=\langle F_0,F_1\rangle$. If $a=2$ then since $F_2 \cap
  F_0 \subset F_1$, the facet $F_2$ must be a leaf with joint $F_1$.

  Now suppose that $a>2$ and the statement is true up to the
  $(a-1)$-st step. If $a-b=1$ then  $$F_0\cap F_1=F_0\cap F_2=\cdots
  =F_0\cap F_a.$$ By \cite{F2} Lemma~4.1 we know that
  $\Gamma$ must have two leaves, and so one of the facets $F_1,\ldots,
  F_a$ is a leaf. 

   We assume that $a-b\geq 2$ and neither one of $F_a,\ldots,F_{a-b}$
   is a leaf of $\Gamma$.  There are two possible cases.

   \begin{enumerate}
    \item The case $b=0$. Then $F_0 \cap F_{a-1} \supsetneq F_0 \cap
      F_a$.  If $\Gamma'=\langle F_0,\ldots,F_{a-2},F_a\rangle$ then
      are two scenarios.
      \begin{enumerate}
      \item If $\Gamma'$ is disconnected,
      then the facet $F_a$ alone is a connected component of $\Gamma'$
      (since all other facets intersect $F_0$) and therefore $F_a$ is
      a leaf of $\Gamma'$ and $F_a \cap F_i=\emptyset$ for
      $i=0,\ldots,a-2$. Since $\Gamma$ is connected, $F_{a-1} \cap
      F_a\neq \emptyset$, and therefore $F_a$ is a leaf of $\Gamma$
      with joint $F_{a-1}$. 

      \item If $\Gamma'$ is connected, we apply the induction
        hypothesis to the tree $\Gamma'$ with good leaf $F_0$. In the
        ordering of the facets of $\Gamma'$, $F_a$ can only be at the
        right end of the sequence (since $F_0\cap F_{a-2} \supsetneq
        F_0 \cap F_a$). So $F_a$ is a leaf of $\Gamma'$ and hence
        there is a joint $F_j\in \{F_0, \dots, F_{a-2}\}$ such that
        $F_a \cap F_k \subseteq F_j$ for all $F_k \in \{F_0, \dots,
        F_{a-2}\}$.

      If $F_a$ is not a leaf of $\Gamma$ then $F_a \cap F_{a-1} \not
      \subseteq F_j$. It also follows that $F_a \cap F_j \not
      \subseteq F_{a-1}$, as otherwise $F_{a-1}$ would be a joint of
      $F_a$. Therefore, we can now apply Lemma~\ref{l:three-indices}
      to the tree $\langle F_j, F_{a-1}, F_a\rangle$ to conclude that
      $F_j \cap F_{a-1} \subseteq F_a$. It follows that $$ F_0 \cap
      F_j \cap F_{a-1} \subseteq F_0 \cap F_a \Longrightarrow F_0 \cap
      F_{a-1} \subseteq F_0 \cap F_a \subsetneq F_{a-1} \cap F_0$$
      which is a contradiction. So $F_a$ has to be a leaf of $\Gamma$
      and we are done.
      \end{enumerate}

   \item The case $b>0$. We keep the good leaf $F_0$ and generate
     complexes $\Gamma_i=\Gamma \rmv{ F_i}$ for $i \in \{
     1,\ldots,a\}$.  By induction hypothesis each $\Gamma_i$ has a
     leaf $F_{u_i}$ where $u_i \in \{ a-b,\ldots,\hat{i},\ldots,
     a\}$. Since there are a total of $b+1$ facets that can be leaves
     of the $\Gamma_i$, and there are $a>b+1$ of the complexes
     $\Gamma_i$ (recall that we are assuming $a-b \geq 2$), we must
     have $u_i=u_j=u$ for some distinct $i,j \in \{
     1,\ldots,a\}$. Suppose $F_{v_i}$ and $F_{v_j}$ are the joints of
     $F_u$ in $\Gamma_i$ and $\Gamma_j$, respectively. So we have
   \begin{align}
    F_u\cap F_h \subseteq F_{v_i}& \mbox{ if } h\neq i \notag\\
    F_u\cap F_h \subseteq F_{v_j}& \mbox{ if } h\neq j.\label{e:u-leaf}
   \end{align}
   
   These two embeddings imply  that
   \begin{align}
    F_u\cap F_j \subseteq F_{v_i}\cap F_u \subseteq F_{v_j} & \mbox{
      if } v_i\neq j\notag\\ 
    F_u\cap F_i \subseteq F_{v_j}
    \cap F_u \subseteq F_{v_i}& \mbox{ if } v_j\neq i.\label{e:u-leaf-general}
   \end{align}

   Suppose $v_i \neq j$. Then from (\ref{e:u-leaf}) and
   (\ref{e:u-leaf-general}) we can see that $F_u$ is a leaf of
   $\Gamma$ with joint $F_{v_j}$. Similarly $F_u$ is a leaf of
   $\Gamma$ if $v_j\neq i$. So $F_u$ is a leaf of $\Gamma$ unless
   $v_i=j$ and $v_j=i$ are the only possible joints for $F_u$ in
   $\Gamma_i$ and $\Gamma_j$, respectively.  In this case
   (\ref{e:u-leaf}) turns into
   \begin{align}
    F_u\cap F_h \subseteq F_{j} & \mbox{ if } h\neq i \notag\\
    F_u\cap F_h \subseteq F_{i} & \mbox{ if } h\neq j. \label{e:ij-form}
   \end{align}

   Now consider $\Gamma_u=\D \rmv{ F_u}$, which by induction hypothesis
   must have a leaf $F_v$ with $v \in \{a-b,\ldots,a\}\sm \{u
   \}$ and a joint $F_t$. Since $F_i, F_j \in \Gamma_u$, we must have
   \begin{align} 
    F_i\cap F_v \subseteq F_t & \mbox{ if } v\neq i \notag \\
    F_j\cap F_v \subseteq F_t & \mbox{ if } v\neq j. \label{e:vt-form}
    \end{align}

    Once again, we consider two cases.
    \begin{enumerate}
    \item If $v$ can be selected outside $\{i,j\}$, we combine
      (\ref{e:vt-form}) with (\ref{e:ij-form}) to get
        $$F_u\cap F_v \subseteq F_{j} \cap F_v \subseteq F_t$$ meaning
      that $F_v$ is a leaf of $\Gamma$.
   
    \item If $v$ must be in $\{i,j\}$, then the only leaves of $\Gamma_u$ are
      $F_i$ and $F_j$.  As $F_0\in \Gamma_u$ is a good leaf of $\D$,
      one of $i$ and $j$ must be $0$, say $j=0$. But now we have $$F_u
      \cap F_j=F_u\cap F_0 \subseteq F_i$$ which together with
      (\ref{e:ij-form}) implies that $F_u$ is a leaf of $\Gamma$ with
      joint $F_i$.
    \end{enumerate}
      \end{enumerate}
   \end{proof}

Our main theorem is now just a direct consequence of
Proposition~\ref{p:glo}, with a bit more added to it.

\begin{theorem}[Main theorem: good leaf orders]\label{t:glo} 
Let $\D$ be a simplicial tree with a good leaf $F_0$. Then there is an
order $F_0,F_1,\ldots,F_q$ on the facets of $\D$ such that
\begin{enumerate}
\item $F_0\cap F_1 \supseteq F_0 \cap F_2
\supseteq \cdots \supseteq F_0\cap F_q$, and 
\item The facet $F_i$ is a leaf of $\D_i=\langle F_0,\ldots, F_i\rangle$ for
  $0 \leq i \leq q$.
\item The facet $F_{i-1}$ is a either a leaf of $\D_i$ with the same
  joint as it has in $\D_{i-1}$, or it is the unique joint of $F_{i}$ in
  $\D_i$, for $1 \leq i \leq q$.
\item $\D_i=\langle F_0,\ldots, F_i\rangle$ is connected for $0 \leq i
  \leq q$.
\end{enumerate}
\end{theorem}

  \begin{proof} The good leaf $F_0$ induces an order on the facets of $\D$ 
   that satisfies the first property.  We need to refine this order to
   achieve the second property. Let $i \in \{1,\ldots,q\}$. Starting
   from the beginning, here is how we proceed. For $i\in
   \{1,\ldots,q\}$ let $c_i$ be the largest nonnegative integer such
   that $F_i \cap F_0=F_{i+c_i}\cap F_0$ where $i+c_i\leq q$. 

   Set $i=1$.

   \begin{enumerate}
    \item[Step 1] If $c_i=0$ then set $i:=i+1$ and go back to Step 1.

    \item[Step 2] If $c_i >0$ then we reorder $F_i,\ldots, F_{i+c_i}$
      as follows. By Proposition~\ref{p:glo} there is a leaf
      $F_{\ell_{c_i}} \in \{F_i,\ldots, F_{i+c_i}\}$ of
      $\Gamma=\langle F_0,\ldots,F_{i+c_i}\rangle$. Applying the same
      proposition again there is a leaf $F_{\ell_{c_i-1}} \in
      \{F_i,\ldots, F_{i+c_i}\} \sm \{F_{\ell_{c_i}}\}$ of $\Gamma
      \rmv{F_{\ell_{c_i}}}$. We continue this way $c_i+1$ times and in
      the end we have a sequence $$F_{\ell_0},F_{\ell_1},\ldots,
      F_{\ell_{c_i}}$$ which is a reordering of the facets
      $F_i,\ldots, F_{i+c_i}$ that satisfies both properties (1) and
      (2) in the statement of the theorem. We relabel $F_i,\ldots,
      F_{i+c_i}$ with this new order and set $i:=i+c_i+1$.
    
    \item[Step 3] If $i>q$ we stop and otherwise we go back to Step 1. 
   \end{enumerate}

   At the end of this algorithm, the facets of $\D$ have the desired
   order.

   To prove the third part of the theorem, note that as $F_{i-1}$ is a
   leaf in $\D_{i-1}$, it has a set of free vertices in $\D_{i-1}$
   which we call $A$. There are two scenarios.

   \begin{itemize}
    \item[-] If $F_i \cap A \neq \emptyset$, then $F_{i-1}$ has to be
      the unique joint of $F_i$ in $\D_i$, as no other facet of $\D_i$
      would contain any element of $A$.

    \item[-] If $F_i \cap A = \emptyset$, then $F_i \cap F_{i-1}
      \subseteq \D_{i-2}\cap F_{i-1} \subseteq F_\alpha$, where
      $F_\alpha$ is the joint of $F_{i-1}$ in $\D_{i-1}$. Therefore,
      $F_{i-1}$ is a leaf of $\D_i$.
   \end{itemize}

   Finally to see that $\D_i$ is connected for every $i$, we consider
   two situations.

    \begin{enumerate}
    \item $F_i\cap F_0 \neq \emptyset$. In this case $\D_i$ is
      connected as all facets of $\D_i$ intersect $F_0$.

    \item $F_i \cap F_0=\emptyset$. If $i=q$ then $\D_i=\D$ which is
      connected. Now we assume that $i$ is the smallest index with
      $F_i \cap F_0=\emptyset$, and $c_i>0$, and we consider how
      $\D_i,\ldots,\D_q=\D$ are built in Step 2.  We start from $\D$,
      and pick a leaf for $\D$ from $F_i, \ldots, F_q$.  We call this
      facet $F_q$ and we know already that $\D_q=\D$ must be
      connected. To pick $\D_{q-1}$ we remove the leaf $F_q$ from
      $\D$, and so $\D_{q-1}$ has to be connected. To build $\D_{q-2}$
      we again remove a leaf from $\D_{q-1}$, which forces $\D_{q-2}$
      to be connected, and so on until we reach $\D_i$, which by the
      same reasoning has to be connected.
   \end{enumerate}
  \end{proof}

%%%%%%%%%%%%%%%%%%%%%%%%%%%%%%%%%%%%%%%%%%%%%%%%%%%%%%%%%%%%%%%%%%%%%%%%%%%%%
%%%%%%%%%%%%%%%%%%%%%%%%%%%%%%%%%%%%%%%%%%%%%%%%%%%%%%%%%%%%%%%%%%%%%%%%%%%%%

\section{The effect of good leaf orders on resolutions}

Recall that for a monomial ideal $I$, the notation $\GG(I)$ denotes
the unique minimal monomial generating set for $I$.

\begin{definition}[Splitting~\cite{EK}]\label{d:splitting} A monomial 
ideal $I$ is called {\bf splittable} if one can write $I=J+K$ for two
nonzero monomial ideals $J$ and $K$, such that
\begin{enumerate}
\item $\GG(I)$ is the disjoint union of $\GG(J)$ and  $\GG(K)$;
\item There is a \emph{splitting function} $\GG(J\cap K) \to \GG(J) \times
  \GG(K)$ taking each $w \in \GG(J\cap K)$ to $(\phi(w), \psi(w))$ satisfying
\begin{enumerate}
\item For each $w \in \GG(J\cap K)$, $w=\lcm (\phi(w), \psi(w))$
\item For each $S\subseteq \GG(J\cap K)$, $\lcm (\phi(S))$ and $\lcm (\psi(S))$ strictly divide $\lcm (S)$.
\end{enumerate}
\end{enumerate}
\end{definition}

If a monomial ideal is splittable, then its Betti numbers can be broken down
into those of sub-ideals.

\begin{theorem}[\cite{EK, Fa}]\label{t:EK} 
  If $I$ is a monomial ideal with a splitting $I=J+K$, then for all
  $i,j \geq 0$
$$\beta_{i,j}(I)=\beta_{i,j}(J) + \beta_{i,j}(K) + \beta_{i-1,j}(J\cap K).$$
\end{theorem}

Our next observation is that a good leaf order on a simplicial tree
provides a basic splitting of its facet ideal.

\begin{theorem}[Splitting using a good leaf order]\label{t:glo-res} 
If $I$ is the facet ideal of a simplicial tree $\D$ with a good leaf
$F_0$ and good leaf order
$$F_0 \cap F_1 \supseteq F_0 \cap F_2 \supseteq \ldots \supseteq F_0
\cap F_t \supsetneq F_0 \cap F_{t+1}=\ldots=F_0 \cap F_q=\emptyset$$ and
$J=(F_0,\ldots,F_t)$ and $K=(F_{t+1},\ldots,F_q)$, then $I=J+K$ is a
splitting of $I$.
\end{theorem}

\begin{proof} It is clear that $I=J+K$. We number the vertices of
  $F_0,\ldots,F_t$ in some order as $x_1,\ldots,x_m$. We will build
  $\phi$ and $\psi$ as in Definition~\ref{d:splitting}. Suppose $L \in
  \GG(J \cap K)$. Then there are facets $F_i$ and $F_j$ such that $i
  \leq t < j$ such that $L=\lcm (F_i,F_j)$. Of all choices of such
  $F_i$ we pick one minimal with respect to lex order and call it
  $G_L$, and there is only one choice for $F_j$ (since each $F_j$ adds
  one or more new vertices to the sequence $F_0,\ldots,F_{j-1}$); call
  this facet $H_L$. So we have $L=\lcm(G_L , H_L)$. Let $\phi(L)=G_L$
  and $\psi(L)=H_L$ so that we have a map 
 
$$\begin{array}{cll}
\GG(J \cap K)& \to &\GG(J) \times \GG(K)\\ 
L &\to &(\phi(L),\psi(L))=(G_L,H_L)
 \end{array}$$

  We only need to show that Condition~(b) in
  Definition~\ref{d:splitting} holds. Suppose $S=\{L_1,\ldots,L_r\}
  \subseteq \GG(J\cap K)$. Suppose, as before, for each $i$ we can
  write $L_i=\lcm(G_{L_i},H_{L_i})=G_{L_i} \cup H_{L_i}$ where
  $G_{L_i} \in \GG(K)$ and $H_{L_i} \in \GG(K)$. We need to show
 
 \begin{enumerate}
  \item $G_{L_1}\cup \cdots \cup G_{L_r} \subsetneq L_1 \cup \cdots \cup L_r$.

    This is clear since each of the $L_i$ contains vertices that are in
    $\GG(K)$ but not in $\GG(J)$.
  
  \item $H_{L_1}\cup \cdots \cup H_{L_r} \subsetneq L_1 \cup \cdots \cup L_r$

    Each of the $L_i$ has a nonempty intersection with $F_0$, but
    $H_{L_i}\cap F_0=\emptyset$, which makes the inclusion above
    strict.
  \end{enumerate}

So we have shown that we have a splitting which completes the proof.   
  \end{proof}

  As a result, we can use good leaf orders to bound invariants related
  to resolutions of trees. Recall that the {\bf regularity} of an
  ideal $I$, denoted by $\reg(I)$, is the maximum value of $j-i$ where
  $\beta _{i,j}(I) \neq 0$. The {\bf projective dimension} of $I$,
  denoted by $\pd(I)$, is the maximum value of $i$ where $\beta
  _{i,j}(I) \neq 0$ for some $j$. The projective dimension and
  regularity measure the ``length'' and the ``width'' of a minimal
  free resolution, as can be seen in the Betti diagram of the ideal;
  see Example~\ref{e:examp} below.  For a simplicial complex $\Gamma$
  we often use the notation $\beta _{i,j}(\Gamma)$, $\reg(\Gamma)$ and
  $\pd(\Gamma)$ to indicate the Betti numbers, regularity and
  projective dimension of $\F(\Gamma)$.

   The following statement is a direct application of
   theorems~\ref{t:EK} and~\ref{t:glo-res}.

  \begin{corollary}[Bounds on Betti numbers of trees]\label{c:bounds} 
Suppose $\D$ is a simplicial tree that can be partitioned into
subcollections $\D_0, \ldots,\D_s$, each of which is a tree, and such
that for each $i=0,\ldots,s$, setting $a_0=0$ we have:
\begin{enumerate}
\item $\D_i=\tuple{F_{a_i},F_{a_i+1},\ldots,F_{a_{i+1}-1}}$ with good leaf $F_{a_i}$. 

\item $F_{a_i+1}\cap F_{a_i}\supseteq \ldots \supseteq
  F_{a_{i+1}-1}\cap F_{a_i} \neq \emptyset$ is a good leaf order on
  $\D_i$;

\item $F_{a_i} \cap F_j = \emptyset$ for $j \geq a_{i+1}$.
\end{enumerate}
Then $$\beta _{i,j}(\D) \geq \beta _{i,j}(\D_0) + \cdots + \beta
_{i,j}(\D_s).$$ In particular
$$\pd(\D) \geq \max \{\pd(\D_0), \ldots, \pd(\D_s)\}$$ and
$$\reg(\D) \geq \max \{\reg(\D_0),\ldots,\reg(\D_s)\}.$$
\end{corollary}

We demonstrate the effect via the example of Figure~\ref{f:glo} which
we will label below.

\begin{example}\label{e:examp} For the ideal $I=(xyz, yzv, yu, vw, wt)$ the facet complex $\D$ is 
\begin{center}
\includegraphics{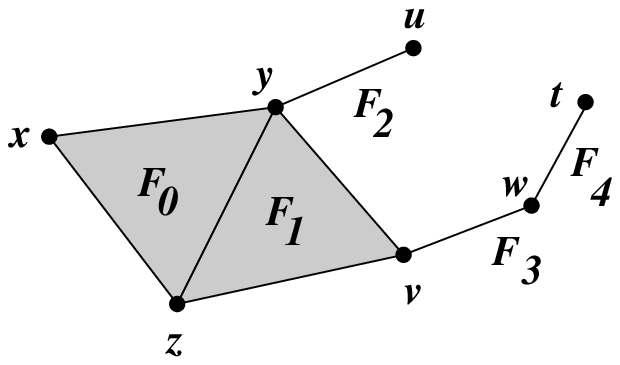}
\end{center}

Following the statement of the corollary, we can find a splitting for $I$ by
partitioning  the facets of $\D$ into two trees with the written good leaf orders
$$ \tuple{F_0, F_1, F_2} \hspace{1cm} \mbox{and} \hspace{1cm}
\tuple{F_3,F_4}$$ which correspond, respectively, to the two ideals
$$ J=(xyz, yzv, yu) \hspace{1cm} \mbox{and}  \hspace{1cm} K=(vw, wt).$$

We copy the Betti diagrams of $I$, $J$ and $K$ (in that
order) using Macaulay2~\cite{M2}.

\bigskip
\begin{center}
\begin{tabular}{|r|llll|}
\hline
          $\mathbf{I}$ \ &0&1&2&3\\
\hline
          1&\emph{3}&\emph{1}&.&.\\
          2&\emph{2}&\emph{6}&\emph{3}&.\\
          3&.&\emph{1}&\emph{2}&\emph{1}\\
\hline
\end{tabular}\hspace{.5in}
\begin{tabular}{|r|lll|}
\hline 
     $\mathbf{J}$\ &0&1&2\\
\hline
     1&\emph{1}&.&.\\
     2&\emph{2}&\emph{3}&\emph{1}\\
\hline
\end{tabular}\hspace{.5in}
\begin{tabular}{|r|ll|}
\hline 
    $\mathbf{K}$\ &0&1\\
\hline
    1&\emph{2}&\emph{1}\\
\hline
\end{tabular}
\end{center}
\bigskip

The bounds presented in Corollary~\ref{c:bounds} are now evident from
the Betti diagrams: $$\reg(I)=3 \geq \max \{ \reg(J),\reg(K) \}=
\max\{2,1\}=2$$ and 
$$\pd(I)=3 \geq \max \{\pd(J),\pd(K)\}=
\max\{2,1\}=2$$
\end{example}

There can be different good leaf orders on a simplicial tree. It would
be interesting to know which one gives a ``better'' splitting, and
better bounds for the resolution invariants.

\subsection{Recursive calculations of Betti numbers}

In~\cite{HV} H\`a and Van Tuyl used Eliahou-Kervaire splittings to
reduce the computation of the Betti numbers of a given simplicial
forest to that of smaller ones. Our goal here is to show that their
formula can be refined in certain cases and be used to
compute the Betti numbers of a given simplicial tree in terms of
intersections of the faces. The method used is essentially a repeated
application of a splitting formula due to H\`a and Van Tuyl~\cite{HV}
to a good leaf order on a given tree, along with an argument that, at
every stage, we know what the next splitting to consider should
be.

\begin{definition}[\cite{HV} Definition~5.1]\label{d:conn} Let $F$ be a facet
 of a simplicial complex $\D$.  The {\bf connected component of $F$ in
 $\D$}, denoted $\conn_\D(F)$, is defined to be the connected component
 of $\D$ containing $F$. If $\conn_\D(F) = \langle G_1,\ldots, G_p
 \rangle$, then we define the {\bf reduced connected component of $F$ in
 $\D$}, denoted by $\oconn_\D(F)$, to be the simplicial complex whose
 facets are a subset of $\{G_1\sm F,\ldots, G_p\sm F\}$, chosen so
 that if there exist $G_i$ and $G_j$ such that $\emptyset \neq G_i\sm
 F \subseteq G_j \sm F$, then we shall disregard the bigger facet $G_j
 \sm F$ in $\oconn_\D(F)$.
\end{definition}

Note that in the Definition~\ref{d:conn}, $\oconn_\D(F)$ is the
localization of $\conn_\D(F)$ at the ideal generated by the complement
of the facet $F$. Therefore if $\D$ is a tree then $\oconn_\D(F)$ is
always a forest~(\cite{F1}). H\`a and Van Tuyl~(\cite{HV}~Lemma~5.7)
prove this directly in their paper.

A facet $F$ of $\D$ is called a {\bf splitting facet of $\D$} if
$\F(\D) = (F) + \F(\D\rmv{F})$ is a splitting of $\F(\D)$ (here we are
thinking of $F$ as a monomial).

\begin{theorem}[\cite{HV} Theorem~5.5]\label{t:HV-main} If $F$ is a 
  splitting facet of a simplicial complex $\D$, then for all $i\geq 1$
  and $j\geq 0$ we
  have 
\begin{align}
\beta_{i,j}(\F(\D))=\beta_{i,j}(\F(\D\rmv{F}))+\sum_{l_1=0}^i\sum_{l_2=0}^{j-|F|}
  \beta_{l_1-1,l_2}(\F(\oconn_\D(F)) \beta_{i-l_1-1,j-|F|-l_2}(\F(\D
  \sm \conn_\D(F)). \label{e:gHV-recursive}
\end{align}
\end{theorem}

So now the question is what is a good choice for a splitting facet. In
their paper (\cite{HV}~Theorem~5.6) H\`a and Van Tuyl show that a leaf
of a simplicial complex is a splitting facet. Their proof in fact only
requires the facet to have a free vertex.

\begin{proposition} Let $\D$ be a simplicial
  complex.  If $F$ is a facet of $\D$ with a free vertex, then $F$ is
  a splitting facet of $\D$.
\end{proposition}

  \begin{proof} The proof is identical to the proof of 
    Theorem~5.6 in~\cite{HV}.
  \end{proof}

We use the convention that for any ideal $I$
\begin{align} \beta_{-1,j}(I)= \left \{
\begin{array}{ll}
1 & j=0\\
0 & \mbox{otherwise}.
\end{array} \right. \label{e:gb-1}
\end{align}

Suppose we have a simplicial tree $\D$ with good leaf order
described as in Theorem~\ref{t:glo}. We apply (\ref{e:gHV-recursive})
to $\D=\langle F_0,\ldots,F_q\rangle$ peeling
off leaves in the following order: $F_q,F_{q-1},\ldots,F_0$.

Suppose we are in step $u$, peeling off the leaf $F_u$ from the tree
$\D_u=\langle F_0,\ldots,F_u\rangle$. Then $\conn_{\D_u}(F_u)=\D_u$
and so $\F(\D_u\sm \conn_{\D_u}(F_u))=0$ and therefore 
$$\beta_{a,b}(\F(\D_u\sm \conn_{\D_u}(F_u))= \left \{
\begin{array}{ll}
1 & a=-1,\ b=0\\
0 & \mbox{otherwise}.
\end{array} \right.$$ 
Applying this to (\ref{e:gHV-recursive}), we solve $i-l_1-1=-1$ and
$j-|F_u|-l_2=0$ to find $l_1=i$ and $l_2=j-|F_u|$. 

Moreover, we have $\oconn_{\D_u}(F_u)=(\D_{u-1})_{\overline{F_u}}$,
that is $\D_{u-1}$ localized at the ideal generated by the complement
of the facet $F_u$ using notation as in Lemma~\ref{l:localization}.
So (\ref{e:gHV-recursive}) turns into

\begin{align*}\beta_{i,j}(\F(\D))&=
\beta_{i,j}(\F(\D_{q-1}))+ \beta_{i-1,j-|F_q|}(\F((\D_{q-1})_{\overline{F_q}}))\notag\\
&= \beta_{i,j}(\F(\D_{q-2}))+ \beta_{i-1,j-|F_{q-1}|}(\F((\D_{q-2})_{\overline{F_{q-1}}}))+ \beta_{i-1,j-|F_q|}(\F({(\D_{q-1})}_{\overline{F_q}}))\notag\\
&\vdots\notag\\
&=\beta_{i,j}(\F(\langle F_0 \rangle))+\sum_{u=1}^{q}\beta_{i-1,j-|F_{u}|}(\F((\D_{u-1})_{\overline{F_{u}}}))
\end{align*}

Note that we did not use the fact that $F_0$ is a good leaf here, just
that each $F_u$ is a leaf if $\D_u$.  We have therefore justified  the
following statement.

\begin{proposition} Let $\D$ be a simplicial tree with a  good 
leaf order $F_0, F_1,\ldots, F_q$ such that each $F_u$ is a leaf of
  $\D_u=\tuple{F_0,\ldots,F_u}$ for $u \leq q$. Then for all $i\geq 1$
  and $j\geq 0$
 \begin{align}
\beta_{i,j}(\F(\D))=\beta_{i,j}(\F(\langle F_0
  \rangle))+\sum_{u=1}^{q}\beta_{i-1,j-|F_{u}|}(\F((\D_{u-1})_{\overline{F_{u}}})).
\label{e:gSF-recursive}
\end{align}
\end{proposition}

 By introducing an appropriate
``$\delta$'' function we can say
\begin{align}
\beta_{i,j}(\F(\langle F_0 \rangle))=\delta_{(i,j),(0,|F_0|)}=
\left\{ \begin{array}{ll}
1 & i=0,\ j=|F_0|\\
0 & \mbox{otherwise}.
\end{array} \right. \label{e:gone-gen}
\end{align}

So now we focus on the structure of $(\D_{u-1})_{\overline{F_u}}$. The
main point that we would like to make is that
$(\D_{u-1})_{\overline{F_u}}$ behaves well, in other words, it satisfies
the same kind of inclusion sequence enforced in Theorem~\ref{t:glo},
and the same ``leaf-peeling'' property. Note that though $F_0$ need
not even survive the localization, its role is that of a virtual glue
that forces facets to always stick together and have an appropriate
order.

\begin{proposition}\label{p:gDu} Let $\D$ be a simplicial 
tree with a good leaf $F_0$ and good leaf order $$F_0\cap F_1
\supsetneq F_0 \cap F_2 \supsetneq \cdots \supsetneq F_0\cap F_q.$$
Suppose $u\in \{1,\ldots,q\}$ and $\D_u=\tuple{F_1,\ldots,F_u}$, and
suppose $(\D_{u-1})_{\overline{F_{u}}}$ has facets $F_{a_1}\sm
F_u$, \ldots, $F_{a_s}\sm F_u$ with $0\leq a_1 <\ldots <a_s\leq u-1$. Then

\begin{enumerate}
\item $a_s=u-1$,
\item $(F_0\cap F_{a_1})\sm F_u \supsetneq \ldots \supsetneq (F_0\cap
  F_{a_s})\sm F_u$,
\item $(\D_{u-1})_{\overline{F_u}}$ is a simplicial tree,
\item If $F_v$ is a joint of $F_{u}$ in $\D_{u}$ then $F_v \sm F_u
  \in (\D_{u-1})_{\overline{F_u}}$,
\item $F_{u-1}\sm F_u$ has a free vertex in $(\D_{u-1})_{\overline{F_u}}$.
\end{enumerate}
\end{proposition}

   \begin{proof} To prove 1, suppose 
    there is an $i<u-1$ such that $(F_i\sm F_u) \subset
    (F_{u-1}\sm F_u)$. By assumption there
    exists $y\in (F_0\cap F_i) \sm (F_0\cap F_{u-1})$. As $(F_0\cap
    F_{u-1}) \supset (F_0\cap F_{u})$, it follows that $y \in (F_i\sm
    F_u)$ and $y \notin (F_{u-1}\sm F_u)$, which contradicts the
    inclusion $(F_i\sm F_u) \subset (F_{u-1}\sm F_q)$.  
 
    The strict inclusions in 2 follow from the same observation, that
    for every $i$ there is always an element in $F_0 \cap F_{a_i}$
      which is not in $F_{a_{i+1}}$ or $F_u$.

    Since $(\D_{u-1})_{\overline{F_u}}$ is a localization of the tree
    $\D_{u-1}$, it is clear that it is a forest, and by 2, since $(F_0\cap
    F_{a_s})\sm F_u\neq \emptyset$, it must be connected and therefore
    a simplicial tree. This settles 3.
 
    For 4, suppose for some $j <u$ we have $F_j\sm F_u \subseteq
    F_v\sm F_u$. Then we will have $$F_j=(F_j\cap F_u) \cup (F_j \sm
    F_u) \subseteq F_v$$ which implies that $F_j=F_v$. 

    Finally to prove 5 we use induction on $u$. If $u=1$ or $2$, then
    $(\D_{u-1})_{\overline{F_u}}$ will have one or two facets, and in
    each case $F_{u-1}\sm F_u$ clearly must have a free vertex. If
    $u=3$ then $F_2$ is a leaf of $\D_2$ with a joint $F_i$ for some
    $i<2$. If $F_i\sm F_3 \in (\D_2)_{\overline{F_3}}$, then it acts
    as a joint of $F_2\sm F_3$ so $F_2\sm F_3$ is a leaf and must
    therefore have a free vertex. If $F_i\sm F_3 \notin
    (\D_2)_{\overline{F_3}}$, then $(\D_2)_{\overline{F_3}}$ has at
    most two facets including $F_2\sm F_3$, each of which must have a
    free vertex. This settles the base cases for induction.

   Now suppose $u\geq 4$ and $F_{u-1}\sm F_u$ has no free vertex in
   $(\D_{u-1})_{\overline{F_u}}$.

   By induction hypothesis, if we consider $\Gamma=\D\sm
   \tuple{F_{u-2}}$, then $F_{u-1} \sm F_u$ will have a free vertex
   $x$ in $(\Gamma_{u-1})_{\overline{F_u}}$. If $x$ is not a free
   vertex in $(\D_{u-1})_{\overline{F_u}}$, then for some $j<u-1$ we have
   $x \in F_{j} \sm F_u \in (\D_{u-1})_{\overline{F_u}}$ and $F_{j}
   \sm F_u \notin (\Gamma_{u-1})_{\overline{F_u}}$. The only possible
   such index $j$ is $j=u-2$. In other words, $x \in F_{u-1} \cap
   F_{u-2}$ and $x \notin F_ i$ for any other $i\leq u$.
   
   Similarly, if we remove $F_{u-3}$ from $\D$ we will find a vertex
   $y \in F_{u-1} \cap F_{u-3}$ and $y \notin F_ i$ for any other
   $i\leq u$.
 
   By Lemma~\ref{l:three-indices}, we must then have
   $F_{u-3}\cap F_{u-2} \subseteq F_{u-1}$. Intersecting both sides
   with $F_0$ we obtain $$F_{u-1} \cap F_0 \subseteq F_{u-2} \cap
   F_0= F_{u-3}\cap F_{u-2} \cap F_0 \subseteq F_{u-1} \cap F_0$$ which
   means that $F_{u-1} \cap F_0 = F_{u-2} \cap F_0$; a
   contradiction.
      \end{proof}

Proposition~\ref{p:gDu} now allows us to continue solving
(\ref{e:gSF-recursive}) by applying Theorem~\ref{t:HV-main} once
again, since we have a splitting facet for each $(\D_{u-1})_{\overline{F_u}}$.
Consider the tree $\D$ as described above with the good leaf order
described in Theorem~\ref{t:glo} and for some $u \in \{1,\ldots,q\}$,
let $(\D_{u-1})_{\overline{F_u}}=\langle F_{a_1}\sm F_u,\ldots,F_{a_s}\sm F_u \rangle
$ where $0\leq a_1< \ldots <a_s <u$.  By Proposition~\ref{p:gDu}
$(\D_{u-1})_{\overline{F_u}}$ is a simplicial tree with an order of the facets
induced by the good leaf order of $\D$, and with splitting facet
$F_{a_s}\sm F_u$.

We continue in the same spirit. Let $u_1=u, u_2=a_s$ and
$$\C_{u_1,u_2}=((\D_{u_1-1})_{F_{u_1}})_{F_{u_2}\sm F_{u_1}} =\langle
F_{d_1}\sm (F_{u_1}\cup F_{u_2}),\ldots,F_{d_w} \sm (F_{u_1}\cup
F_{u_2})\rangle$$ where $0 \leq d_1<\ldots<d_w<u_2<u_1$.

Similarly, we can build $\C_{u_1,\ldots,u_m}$ which is the
localization of
\begin{align}\label{e:star}
C_{u_1,\ldots,u_{m-1}}=\langle F_{c_1}\sm (F_{u_1} \cup \ldots \cup
F_{u_{m-1}}),\ldots,F_{c_r}\sm (F_{u_1} \cup \ldots \cup F_{u_{m-1}})
\rangle
\end{align}
at the ideal generated by the complement of the facet
$F_{u_m} \setminus (F_{u_1} \cup \ldots \cup F_{u_{m-1}})$ where
$u_m=c_r$.  So we have
\begin{align}\label{e:double-star}
\C_{u_1,\ldots,u_{m}}=\langle F_{b_1}\sm (F_{u_1} \cup \ldots
\cup F_{u_{m}}),\ldots, F_{b_t} \sm (F_{u_1} \cup \ldots \cup
F_{u_{m}}) \rangle
\end{align}
where $b_1,\ldots,b_t\in \{c_1,\ldots,c_{r-1}\}$,
and $$0 \leq b_1<b_2<\ldots<b_t <c_r=u_m<u_{m-1}<\ldots<u_1.$$

\begin{proposition}\label{p:gDug} Let $\D$ be a simplicial 
tree with a good leaf $F_0$ and good leaf order $$F_0\cap F_1
\supsetneq F_0 \cap F_2 \supsetneq \cdots \supsetneq F_0\cap F_q.$$ With
notation as in~(\ref{e:star}) and~(\ref{e:double-star}) above, we
have
\begin{enumerate}

\item $b_t=c_{r-1}$,

\item $(F_0\cap F_{b_1})\sm (F_{u_1} \cup \ldots
  \cup F_{u_{m}})\supsetneq \ldots \supsetneq (F_0\cap F_{b_t})\sm
  (F_{u_1} \cup \ldots \cup F_{u_{m}})$,

\item $\C_{u_1,\ldots,u_{m}}$ is a simplicial tree,

\item $F_{b_t}\sm (F_{u_1} \cup \ldots \cup
F_{u_{m}})$ has a free vertex in $\C_{u_1,\ldots,u_{m}}$
  and is therefore a splitting facet of $\C_{u_1,\ldots,u_{m}}$.

\end{enumerate}
\end{proposition}

  \begin{proof} Let $A=F_{u_1} \cup \ldots \cup F_{u_m}$. To show 1, 
   suppose there is an $i<r-1$ such that $F_{c_i}\sm A \subset
   F_{c_{r-1}}\sm A$. By the strict inclusions assumed there exists
   $y\in (F_0\cap F_{c_i}) \sm (F_0\cap F_{c_{r-1}})$. As $$F_0\cap
   F_{c_{r-1}} \supsetneq F_0\cap F_{u_m}\supsetneq \ldots \supsetneq
   F_0\cap F_{u_1},$$ it follows that $y \in F_{c_i}\sm A$ and $y
   \notin F_{c_{r-1}}\sm A$, which is a contradiction.

   For 2 it is easy to see that $$(F_0\cap F_{b_1})\sm A \supseteq
   \ldots \supseteq (F_0\cap F_{b_{t}})\sm A.$$ To show that these
   inclusions are strict pick $1\leq i<j < t$, we know that $$F_0\cap
   F_{b_i} \supsetneq F_0\cap F_{b_j} \supsetneq F_0\cap
   F_{u_m}\supsetneq F_0\cap F_{u_{m-1}}\supsetneq \ldots \supsetneq
   F_0\cap F_{u_1},$$ and therefore there exists $y \in (F_{b_i}\cap
   F_0)\sm (F_{b_j}\cup F_{u_1} \cup \ldots \cup F_{u_{m}})$, which
   means that $y \in (F_0\cap F_{b_i})\sm (F_{u_1} \cup \ldots \cup
   F_{u_{m}})$ and $y \notin (F_0\cap F_{b_j})\sm (F_{u_1} \cup \ldots
   \cup F_{u_{m}})$, proving 2.

    Suppose
    $\Omega=\tuple{F_{\omega_0},F_{\omega_1},\ldots,F_{\omega_p}}$ is
    the subcollection of $\D$ consisting of those facets that
    are not contained in $A$ with $$0={\omega_0} <
    {\omega_1}< \ldots <{\omega_p}.$$ Because of the
    strict good leaf order $\Omega$ is a connected forest and
    hence a tree.

    We claim that $C_{u_1,\ldots,u_m}$ is the localization of the tree
    $\Omega$ at the ideal generated by $\overline{A}$. This follows
    from two observations. One is that if at the $i$th step when
    building $C_{u_1,\ldots,u_m}$ there are facets $F_\alpha, F_\beta
    \in \D$ not containing $F_{u_1}\cup \ldots\cup F_{u_i}$, then
    $F_\alpha, F_\beta$ do not contain $A$ and therefore are also
    facets of $\Omega$. Moreover if $F_\alpha\sm (F_{u_1}\cup
    \ldots\cup F_{u_i}) \subseteq F_\beta \sm (F_{u_1}\cup \ldots\cup
    F_{u_i})$, then $F_\alpha\sm A \subseteq F_\beta \sm A$ and
    therefore we can conclude that $C_{u_1,\ldots,u_m}$ is a
    localization $\Omega$ and $\{b_1 \ldots b_t \} \subseteq
    \{{\omega_0}, \ldots, {\omega_p} \}.$

    So $C_{u_1,\ldots,u_m}$ must be a forest, and since it is
    connected by 2, it must be a simplicial tree. This settles 3.

    By the discussion above we can assume $\omega_p=b_t$ and we will
    still have $C_{u_1,\ldots,u_m}$ is a localization of
    $\Omega$. Also note that $F_0=F_{\omega_0}$ is a good leaf of
    $\Omega$ with a strict good leaf order induced by that on $\D$.

   To prove 4 we use induction on $p$.  If $p=1$ or $2$ then
   $C_{u_1,\ldots,u_m}$ will have one or two facets, and in each case
   $F_{b_t}\sm A$ clearly must have a free vertex. If $p=3$ then
   $F_{\omega_2}$ is a leaf of
   $\Omega_{\omega_2}=\tuple{F_{\omega_0},F_{\omega_1},F_{\omega_2}}$
   with a joint $F_{w_i}$ for some $i<2$. If $F_{\omega_i}\sm A \in
   C_{u_1,\ldots,u_m}$, then it acts as a joint of $F_{\omega_2}\sm A$
   so $F_{\omega_2}\sm A$ is a leaf and must therefore have a free
   vertex. If $F_{\omega_i}\sm A \notin C_{u_1,\ldots,u_m}$, then
   $C_{u_1,\ldots,u_m}$ has at most two facets including
   $F_{\omega_2}\sm A$ each of which must have a free vertex. This
   settles the base cases for induction.

   Now suppose $p\geq 4$ and $F_{b_t}\sm A$ has no free vertex in
   $C_{u_1,\ldots,u_m}$.

   By the induction hypothesis, if we consider $\Gamma=\Omega \sm
   \tuple{F_{\omega_{p-1}}}$ then $F_{\omega_p}$ will have a free vertex $x$ in
   $\Gamma_{\overline{A}}$. If $x$ is not a free vertex in
   $\Gamma_{\overline{A}}$ then $x \in F_{\omega_{p-1}} \sm A \in
   \Gamma_{\overline{A}}$. In other words, $x \in F_{\omega_p} \cap
   F_{\omega_{p-1}}$ and $x \notin F_ {\omega_i}$ for any other $i\leq {p}$.
   
   Similarly, if we remove $F_{\omega_{p-2}}$ from $\Omega$ we will find a vertex
   $y \in F_{\omega_p} \cap F_{\omega_{p-2}}$ and $y \notin F_ i$ for any other
   $i\leq {p}$.
 
   By Lemma~\ref{l:three-indices}, we must then have
   $F_{\omega_{p-2}}\cap F_{\omega_{p-1}} \subseteq
   F_{\omega_p}$. Intersecting both sides with $F_0$ we
   obtain $$F_{\omega_p} \cap F_0 \subseteq F_{\omega_{p-1}} \cap F_0=
   F_{\omega_{p-2}}\cap F_{\omega_{p-1}} \cap F_0 \subseteq
   F_{\omega_p} \cap F_0$$ which means that $F_{\omega_p} \cap F_0 =
   F_{\omega_{p-1}} \cap F_0$; a contradiction. This proves 4 and we are done.
  \end{proof}

  Proposition~\ref{p:gDug} replaces Proposition~\ref{p:gDu} as a more
  general version. Back to (\ref{e:gSF-recursive}), we start computing
  Betti numbers of $\F(\D)$ for a given tree $\D$ with good leaf $F_0$
  and strict good leaf order
   $$F_0\cap F_1 \supsetneq F_0 \cap F_2 \supsetneq \cdots \supsetneq
  F_0\cap F_q.$$ The formula
 \begin{align*}\beta_{i,j}(\F(\D))=\beta_{i,j}(\F(\langle F_0 \rangle))+\sum_{u=1}^{q}\beta_{i-1,j-|F_{u}|}(\F((\D_{u-1})_{\overline{F_u}}))%\label{e:number-this}
\end{align*}
  becomes recursive, since in each step after localization we
  again have a simplicial tree with a strict induced order on the facets
  where the last facet remaining is a splitting facet.

  To close, we apply the formula to examine some low Betti numbers.

  Let $i=0$. By (\ref{e:gSF-recursive}) and (\ref{e:gb-1}) we have
  \begin{align*}
\beta_{0,j}(\F(\D))=\sum_{u=0}^{q}\delta_{j,|F_u|}.%\label{e:gcase0}
\end{align*}

Let $i\geq 1$. Because of (\ref{e:gSF-recursive}) and (\ref{e:gone-gen})
we can write
\begin{align}\beta_{i,j}(\F(\D))&
=\sum_{u=1}^{q}\beta_{i-1,j-|F_{u}|}(\F((\D_{u-1})_{\overline{F_u}}))
\label{e:gcase1-prelim}\end{align}

From Proposition~\ref{p:gDug} and (\ref{e:gcase1-prelim}) we can see
that we need the generators of each $\D_u$ in order to produce a
formula for the first graded Betti numbers. To this end, we start from
$\D_u=\langle F_0,\ldots,F_u\rangle$ so that
\begin{align*}
(\D_{u-1})_{\overline{F_u}}&=\langle F_i\sm F_u \st 0\leq i <u \mbox{
    and } (F_j\sm F_u) \not \subseteq (F_i\sm F_u) \mbox{ for } j\neq
  i\rangle\\ &\\ &=\langle F_i\sm F_u \st 0\leq i <u \mbox{ and }
  \frac{\lcm(F_j,F_u)}{F_u}\ndiv \ \frac{\lcm(F_i,F_u)}{F_u} \mbox{
    for } j\neq i\rangle\\ &\\ &=\langle F_i\sm F_u \st 0\leq i <u
  \mbox{ and } \lcm(F_j,F_u)\ndiv \ \lcm(F_i,F_u) \mbox{ for } j\neq
  i\rangle\\
\end{align*}

So we can make our ``delta-function'' to have the $\lcm$ condition
built into it. We define
$$\delta_{a,(b,c)}=\left \{ 
\begin{array}{ll}
1&a=|F_b|,\ \lcm(F_d,F_c) \ndiv \ \lcm(F_b,F_c) \mbox{ for } 0\leq d<c\\
0&\mbox{otherwise}
\end{array}\right.$$

So (\ref{e:gcase1-prelim}) becomes
\begin{align*}\beta_{1,j}(\F(\D))&
=\sum_{u=1}^{q}\beta_{0,j-|F_{u}|}(\F((\D_{u-1})_{\overline{F_u}}))\notag\\
&=\sum_{u=1}^{q} \sum_{\tiny{F \mbox{ facet of } (\D_{u-1})_{\overline{F_u}}}}\delta_{j-|F_u|,|F|}\notag\\
&=\sum_{u=1}^{q} \sum_{v=0}^{u-1}\delta_{j-|F_u|,(v,u)}
\label{e:gcase1}
\end{align*}

By building appropriate delta functions, one can continue in this
manner to build further Betti numbers based on the $\lcm$s of the
facets.

%%%%%%%%%%%%%%%%%%%%%%%%%%%%%%%%%%%%%%%%%%%%%%%%%%%%%%%%%%%%
%%%%%%%%%%%%%%%%%%%%%%%%%%%%%%%%%%%%%%%%%%%%%%%%%%%%%%%%%%%%

\end{document}